\documentclass[a4paper,12pt]{amsart}

\usepackage{amssymb,amsbsy,amsmath,amsfonts,amssymb,amscd}
\usepackage{latexsym}
\usepackage{graphics}
\usepackage{color}
\usepackage{comment}
\input xy
\xyoption{all}

\theoremstyle{plain}
\newtheorem{thm}{Theorem}[section]

\newtheorem{lemma}[thm]{Lemma}

\theoremstyle{definition}
\newtheorem{remark}[thm]{Remark}

\newtheorem{defin}[thm]{Definition}

\newtheorem{example}[thm]{Example}

\numberwithin{equation}{section}
\newcommand{\sG}{{\mathcal G}}

\newcommand{\sT}{{\mathcal T}}

\newcommand{\PP}{\ensuremath{\mathbb{P}}}

\newcommand{\CC}{\ensuremath{\mathbb{C}}}
\newcommand{\RR}{\ensuremath{\mathbb{R}}}
\newcommand{\ZZ}{\ensuremath{\mathbb{Z}}}
\newcommand{\QQ}{\ensuremath{\mathbb{Q}}}

\newcommand{\hol}{\ensuremath{\mathcal{O}}}

\newcommand\la{\lambda}

\newcommand\ze{\zeta}
\newcommand\al{\alpha}
\newcommand\be{\beta}

\newcommand\ga{\gamma}
\newcommand\de{\delta}
\newcommand\e{\epsilon}

\newcommand{\Lam}{\Lambda}

\newcommand{\FF}{\ensuremath{\mathbb{F}}}

\newcommand{\ra}{\ensuremath{\rightarrow}}

\def\eea{\end{eqnarray*}}
\def\bea{\begin{eqnarray*}}

\newcommand\dual{\mathrel{\raise3pt\hbox{$\underline{\mathrm{\thinspace d
\thinspace}}$}}}
\newcommand\qe{\ifhmode\unskip\nobreak\fi\quad $\Box$}       % box for QED

\def\BOX{\hfill\lower.5\baselineskip\hbox{$\Box$}}
\newtheorem{theo}{Theorem}[section]

\newtheorem{remarkk}[theo]{Remark}

\newtheorem{prop}[theo] {Proposition}
\newtheorem{cor}[theo]{Corollary}

\newenvironment{ex}{\begin{example}\rm}{\end{example}}

\def\verde{\color[rgb]{.2,.5,.1}}

\usepackage{hyperref}
\title [Cyclic Symmetry - BDF Manifolds]{Cyclic Symmetry on  Complex Tori and Bagnera-De Franchis Manifolds}

\author{Fabrizio Catanese}
\address {Lehrstuhl Mathematik VIII\\
Mathematisches Institut der Universit\"at Bayreuth\\
NW II,  Universit\"atsstr. 30\\
95447 Bayreuth}
\email{Fabrizio.Catanese@uni-bayreuth.de}

\thanks{
The present work took place in the framework of the 
 ERC Advanced grant n. 340258, `TADMICAMT' }

\date{\today}

\begin{document}

\maketitle

\begin{abstract}
We describe the possible linear actions of a cyclic group $G = \ZZ/n$ on a complex torus, using the cyclotomic
exact sequence for the group algebra $\ZZ[G]$. The main application is devoted to
a structure theorem for Bagnera-De Franchis Manifolds, but we also give an application to hypergeometric integrals.

%We show that Bagnera-De Franchis Manifolds of a fixed dimension are not a bounded family: only up to isogeny.

\end{abstract}

%\begin{dedication}
%In beloved  memory of Paolo (De Bartolomeis)
%\end{dedication}
\addtocontents{toc}{\protect\setcounter{tocdepth}{1}}

\tableofcontents
%=========================================================================
\newpage

\section*{Introduction}

Classically, the word 'hyperelliptic' was used for two different ways of generalizing the class of elliptic curves, i.e. the complex tori of dimension 1.

Hyperelliptic curves are defined to be the curves who admit a map to $\PP^1$ of degree 2, and are not Hyperelliptic Varieties according to
the definition of  the French school of Appell, Humbert, Picard, Poincar\'e. 

The French school     defined the Hyperelliptic Varieties as those smooth projective varieties whose universal covering is
biholomorphic to $\CC^g$ (in particular the Abelian varieties are in this class). For $g=1$ these are just the elliptic curves, whereas a prize, the  Bordin prize, was offered for those mathematicians who would achieve the classification of the Hyperelliptic varieties of dimension  $2$.
 Enriques  and Severi were awarded the Prize in 1907 (\cite{es}), but they withdrew their first paper after discussion with De Franchis (replacing it with a second one); 
 Bagnera and De Franchis were awarded the Bordin Prize in 1909, they gave a simpler proof  (\cite{bdf}) apart of a small gap;  perhaps for this reason 
we prefer to call Bagnera-De Franchis surfaces the Hyperelliptic surfaces which are not Abelian surfaces.

 Kodaira \cite{kod} showed that if we
take the wider class of compact complex manifolds of dimension $2$ whose universal covering is
$\CC^2$, then there are other non algebraic and non K\"ahler surfaces, called nowadays Kodaira surfaces.

Based on Kodaira's work, Iitaka conjectured that if a compact  K\"ahler Manifold $X$ has universal covering 
biholomorphic to $\CC^g$, then necessarily  $X$ is a quotient $ X = T / G$ of a complex torus $T$ by the free action of a finite group $G$ (which we may assume to contain no translations).

The conjecture by Iitaka had been proven in dimension $2$ by Kodaira, and was much later proven  in dimension $3$ by Campana and Zhang \cite{cz}.
Whereas it was shown in \cite{chk} that,  if the abundance conjecture holds, then a   projective smooth variety $X$ with universal covering $\CC^n$ is a Hyperelliptic variety according to the following definition.

\begin{defin}
A Hyperelliptic Manifold $X$  is defined to be a quotient $ X = T / G$ of a complex torus $T$ by the free action of a finite group $G$ which contains no translations. 

We say that $X$ is a  Hyperelliptic Variety if moreover the torus $T$ is projective, i.e., it is an Abelian variety $A$.
\end{defin}

 If the group $G$ is a cyclic group $\ZZ/n$, then such a quotient is called (\cite{bcf}, \cite{topmethods}) a Bagnera-De Franchis manifold
(in dimension $g=2$, $G$ is necessarily cyclic, whereas in dimension $ g \geq 3$ the only examples with $G$ non Abelian have 
$G = D_4$ and were classified in \cite{cd} (for us $D_4$ is the dihedral group  of order $8$).

Indeed, (see for instance \cite{ccd}) every Hyperelliptic Manifold is a deformation of a Hyperelliptic Variety, so that
a posteriori the two notions are related to each other, in particular  the set of underlying differentiable manifolds is the same.

By the so-called Bieberbach's third theorem \cite{bieb1,bieb2} concerning the finiteness of Euclidean cristallographic groups,
Hyperelliptic manifolds of a fixed dimension $g$ belong to a finite number of families.

We give in this paper an explicit boundedness result for the case of  Bagnera-De Franchis Manifolds, for which $G$ is a finite cyclic group
(Theorem \ref{BdF}).

The theory of Bagnera-De Franchis Manifolds was introduced in \cite{bcf}, and also expounded in \cite{topmethods} (following ideas introduced
 first in \cite{catcil}) and our main motivation here was to
expand and improve the presentation given there, whereas we refer the reader to \cite{dem} for a classification of 
Bagnera-De Franchis Manifolds of low dimension.
   
 In order to do so, it is necessary to treat linear actions of a finite group $G$ on a complex torus, and we do this here for
 the case of a cyclic group. The subtle point is to describe these actions not only up to isogeny, but determining explicit 
 the torsion subgroups involved in these isogenies.
 
 The starting  point is that a linear action of a group $G$ on a complex torus consists in two steps:
 
 1) viewing $\Lam : = H_1(T, \ZZ)$ as a $\ZZ[G]$-module
 
 2) choosing an appropriate Hodge decomposition on $\Lam \otimes \CC : = \Lam \otimes_{\ZZ} \CC$
  $$\Lam \otimes \CC = H^{1,0} \oplus \overline{H^{1,0} }$$ which
 is invariant for the group action (i.e., $H^{1,0}$ is a $G$-invariant subspace). 
 
 While 2) uses, for $G$ cyclic, just the eigenspace decomposition,  1) requires us to explain in detail some  elementary
 and mostly well known facts about the group algebra of a cyclic group $G$, $\ZZ[G] = \ZZ[x] / (x^n-1)$.
 
 This is derived in section 3 from an elementary generalization of the Chinese remainder theorem (Theorem \ref{Factorial}),
  concerning quotients of factorial rings by principal ideals (these   appear naturally
 in the intersection theory of divisors), and from some classical results about resultants of cyclotomic polynomials,
 explained in section 2. The interesting result for our purposes is Proposition \ref{direct sum}.
 
 The first application that we give is related to hypergeometric integrals: we calculate explicitly 
 the homology of a cyclic covering of degree $n$ of  the projective line $\PP^1$ as a $\ZZ[\ZZ/n]$-module, under the assumption that there
 is a point of full ramification. The description is particularly nice for the case where there  are two points of full ramification,
 Theorem \ref{2totram}   shows that we have   a direct sum of certain cyclic modules naturally associated to the ramification indices. 
 
 \newpage
 
 {\bf Theorem 4.3}{\em
 
Assume that we have a cyclic covering $ f : C \ra \PP^1$ with group $\ZZ/n$ and with two points of full ramification.

Then, if the order of the inertia groups are $n,r_1, \dots , r_k, n$,  the $\ZZ[x] $-module  $H_1(C, \ZZ)$ is a direct sum of
cyclic  modules, 
$$H_1(C, \ZZ) = \oplus_1^k   \ZZ[x] / ( 1 + x^{n/r_j}  + x^{2n/r_j} + \dots + x^{(r_j-1)n/ r_j}) .$$
}

In the final section 6 we determine also explicitly  the intersection product for the first homology group $H_1(C, \ZZ)$.

  We pose the question of finding
a simple description  in the general case.
 
 \medskip
 
 Section 5 is devoted to the second  and main application, namely, the algebraic description of Bagnera-De Franchis Manifolds
 with group $\ZZ/n$.
 
 The main result is Theorem \ref{BdF}.

{\bf Theorem 5.1}{\em

A Bagnera-De Franchis Manifold with group $G = \ZZ/n$  is completely determined by the following data:

\begin{enumerate}
\item
 the datum of 
torsion free $R_d$-modules $\Lam_d$ of finite rank, for all $d | n$, such that  $\Lam_1 \neq 0$ and with $\Lam_1, \Lam_2$ of even rank;
\item
the datum of a finite subgroup $\Lam^0 \subset A' : =   \oplus _{d|n} A_d$, where $A_d : = \Lam_{d, \RR}/ \Lam_d$,
\item
an element $\be_1 \in A_1$  generating a subgroup $\langle \be_1 \rangle$ of order exactly $n$,
such that:
\item
(A) $\Lam^0$ is stable for multiplication by the element $x$ of the subring $R(n) \subset R'(n) : =  \oplus _{d|n} R_d $,
and
\item
(B) $\Lam^0 \cap A_d = 0$ $\forall d|n$,
\item
(C) the projection of $\Lam^0$ into $A_1$   intersects  the subgroup $\langle \be_1 \rangle$ only in $0$;
\item
the datum of a complex structure on each $\Lam_d \otimes \CC $,  i.e., a Hodge decomposition
$$ \Lam_d \otimes \CC  = V(d) \oplus \overline{V(d)}, $$
which allows to decompose $V(d) = \oplus_{j < d, (j.d)=1} V_j$ as a direct sum of eigenspaces for the action $\al$ of $x$.
\item
The properties (A) and (B) imply  that $\Lam^0 \subset \oplus_{d | n}  ( \frac{\Phi_d}{Q_n} \Lam_d ) / \Lam_d $,
hence, in particular,  the number of such subgroups $\Lam^0$ is finite.

\end{enumerate}

}

 %Using this structure theorem we show in Theorem \ref{unbounded} the unboundedness of Bagnera-De Franchis manifolds of a given dimension.

\bigskip

 Throughout the paper we have been trying to illustrate the concepts introduced, or discussed, via many concrete examples.

\section{An exact sequence in factorial rings}

\begin{thm}\label{Factorial}
	Let  $\hol$ be a factorial ring, and assume that we have an integer $k \geq 2$ and elements $f_1, \dots, f_k \in \hol$, such that
	
	(1) $f_i$ is not a unit 
	
	(2) for $i \neq j$, $f_i$ and $f_j$ are relatively prime.
	
	Then we have a natural exact sequence 
	
	$$ 0 \ra R: = \hol / (f_1 f_2 \dots f_k)  \ra \oplus_1^k \hol / (f_i)  \ra \oplus_{i <  j} \hol / (f_i, f_j)  \ra 0,$$
	
	where, setting $R_i : = \hol / (f_i) $, $R_{i,j} : =  \hol / (f_i, f_j) $ for $i \neq j$,
	
	(3) $ R \ra \oplus_1^k R_i $ is induced by the natural surjections $ R \ra R_i  = \hol / (f_i) $, and where 
	
	(4) $(a_i ) \in  \oplus_1^k R_i \mapsto (a_i - a_j) \in \oplus_{i  <  j} \hol / (f_i, f_j) .$
\end{thm} 

\begin{remark}
Observe that the hypothesis that $f_i$ is not a unit is not really needed, since, if $f_i$ is a unit, then  $R_i = 0 = R_{i, j} \ \forall j$,
and we have the same exact sequence as the one corresponding to the set obtained  from the set $\{f_1, \dots, f_k \}$
by deleting $f_i$.
\end{remark}

The proof follows essentially by induction from the standard special case $k=2$:

\begin{lemma}\label{k=2}
Assume either that 

(I)    $\hol$ is a factorial ring, and    $ f,g \in \hol$ are  relatively prime elements. 

Or assume that

(II) $\hol$ is any  ring,  the ideal $(f)$ is prime, $ g\notin (f)$.
	
	Then we have a natural exact sequence 
	
	$$ 0 \ra R: = \hol / ( f g )  \ra  \hol/(f) \oplus  \hol / (g)  \ra  \hol / (f, g)  \ra 0.$$

\end{lemma}

\begin{proof}
We first prove the exactness of 
$$  \hol  \ra  \hol/(f) \oplus  \hol / (g)  \ra  \hol / (f, g)  \ra 0.$$
Surjectivity is obvious, whereas if $ ( a \ (mod \  f), b\ (mod \ g) )$ maps to $0$, then
$$a-b  \in  (f, g) \leftrightarrow \exists \al, \be \in \hol , a - b = \al f + \be g \leftrightarrow a - \al f = b + \be g : = c.$$ 

But then $c -a \in (f), c - b \in (g)$, proving exactness in the middle.

Finally,  $c \in \hol \mapsto (0,0)$ if and only if $ c \in (f)\cap (g)$.

In case (I), by unique factorization and since $f,g$ are relatively prime, $ c$ is divisible by $fg$, hence the kernel of the first homomorphism is
the principal ideal $(fg)$.

In case (II),  $c \in (g)$ implies the existence of $d$ such that  $c = dg$. Since $c \in f$, and $(f)$ is prime, then necessariy either $g \in (f)$,
or $ d \in (f)$.  The first possibility is excluded by our assumption, therefore there exists $ e \in \hol$ with $ d = ef $, hence $ c = e fg$, and we are done.

\end{proof}

\begin{cor}\label{intermediate}
	Let  $\hol$ be a factorial ring, and assume that we have an integer $k \geq 2$ and elements $f_1, \dots, f_k \in \hol$, such that
	
	(1) $f_i$ is not a unit 
	
	(2) for $i \neq j$, $f_i$ and $f_j$ are relatively prime.
	
	Then, setting $ F : = f_1 f_2 \dots f_k$, the cokernel $N$ of the following exact sequence:  
	
	$$ 0 \ra R: = \hol / (F)  \ra \oplus_1^k \hol / (f_i)  \ra N  \ra 0$$
	
	has a filtration 
	
	$$ 0 :  = N_0 \subset N_1 \subset  \dots \subset N_{k-1} = N$$
	such that $ N_i / N_{i-1} \cong \hol / (f_i, f_{i+1} \dots f_k)$.
	
	\end{cor} 
	
\begin{proof}

Observe first of all that $R \ra \oplus_1^k R_i$ is an inclusion, since our elements $f_i$ are relatively prime.

We prove now  the main assertion by induction on $k$, the case $ k=2$ being the content of lemma \ref{k=2}.

 Set  $g : = f_2 \dots f_k$,  and observe that $F = f_1 g $: by  lemma \ref{k=2} we have an exact sequence
 
$$ 0 \ra  R = \hol / ( F  )  \ra  \hol/(f_1) \oplus  \hol / (f_2 \dots f_k)  \ra  \hol / (f_1, f_2 \dots f_k)  \ra 0.$$

By induction we have an exact sequence 

$$ 0 \ra R': = \hol / (f_2 \dots f_k)  \ra \oplus_2^k \hol / (f_i)  \ra N'  \ra 0,$$
and a filtration $N'_0 \subset \cdots N'_{k-1} = N'$ with the desired properties.
 
Hence we have inclusions

$$ 0 \ra R \ra R_1 \oplus R' \ra \oplus_1^k R_i$$

and, defining $N_1 : =  \hol / (f_1, f_2 \dots f_k) $,  we have   $N' = N / N_1$, and it suffices to define
$ N_i $, for $ i \geq 2$,  to be the inverse image of $N'_{i-1}$ inside $N$.

\end{proof}

\begin{lemma}\label{lemma2}
Let $f,h,g \in \hol$ and assume either that

(i) the ring $\hol$ is factorial, and $f,g$ are relatively prime, or

(ii) the ideal $(f)$ is a prime ideal, and $g \notin (f)$.

Then we have the exact sequence, where the first map is given by multiplication by $g$:

$$ 0 \ra \hol/ (f,h) \ra \hol / (f, gh) \ra  \hol / (f, g) \ra 0.$$

\end{lemma}

\begin{proof}

Set $A : = \hol / (f)$. 

Then we have an exact sequence

$$ A \ra A / (gh) \ra A / (g) \ra 0 ,$$

and the kernel of the first map is the principal ideal generated by $h$,
since 
$$\{ \phi \in A | \phi g  \in (gh) \} = \{ \phi | \exists \psi \ \phi g = gh \psi \} =  \{ \phi | \exists \psi \ \phi  = h \psi \},$$
because $g$ is not a zero divisor in $A$. 

\end{proof}

\begin{lemma}\label{cor3}
	Let  $\hol$ be a factorial ring, and assume that we have an integer $k \geq 2$ and elements $f_1, \dots, f_k \in \hol$, such that
	
	(1) $f_i$ is not a unit 
	
	(2) for $i \neq j$, $f_i$ and $f_j$ are relatively prime.
	
	Then $ \hol / (f_1, f_2 \dots f_k)$ has a filtration whose graded quotient is
	$$ \oplus_{j=2}^k  \ \hol / (f_1, f_j). $$ 
		
	\end{lemma} 
	
\begin{proof}
Apply lemma \ref{lemma2} to $f : = f_1$,  $h : = f_k$, and $g : = f_2 \dots f_{k-1}$, and use induction.

\end{proof}

{\em Proof of Theorem \ref{Factorial}}
\smallskip

We first of all observe that the map $ \oplus_1^k R_i \ra M : =  \oplus_{i < j} R_{i,j}$ factors through the quotient $N$. 

We have shown that $N$ has a filtration whose associated graded ring is exactly isomorphic to $ \oplus_{i < j} R_{i,j}$,
and from this we shall derive   an isomorphism of 
$N$ with  $ \oplus_{i < j} R_{i,j}$.

In fact, by induction, the homomorphism $N \ra M$ induces  an isomorphism $ N' \cong \oplus_{i <  j, i,j \geq 2} R_{i,j} : = M'.$

Moreover, by the definition of the map, $R_1 \oplus R' $ maps to zero inside $M'$.

Hence it suffices to show that $ (R_1 \oplus R' ) / R = N_1 $ maps isomorphically to $M_1 : = \oplus_{ j \geq 2} R_{1,j}$,
and this follows again by corollary \ref{cor3} observing that the homomorphism preserves 
the corresponding filtrations on both modules $M_1, N_1$ and that the homomorphism induces an isomorphism of associated graded modules:
since, by induction (changing the order of the summands), for each $ j$, $N_1$ surjects onto $R_{1,j}$. 

Hence  this homomorphism induces an isomorphism $N_1 \cong M_1$ and the proof is finished.

\qed

We record here a result shown in the course of the proof of Theorem \ref{Factorial}:

\begin{cor}\label{cor4}
	Let  $\hol$ be a factorial ring, and assume that we have an integer $k \geq 2$ and elements $f_1, \dots, f_k \in \hol$, such that
	
	(1) $f_i$ is not a unit 
	
	(2) for $i \neq j$, $f_i$ and $f_j$ are relatively prime.
	
	Then the $\hol$-module $ \hol / (f_1, f_2 \dots f_k)$ is isomorphic to 
		$$ \oplus_{j=2}^k  \ \hol / (f_1, f_j). $$ 
		
	\end{cor} 

\section{An arithmetic application}

In this section we consider the factorial  ring $\hol : = \ZZ [x]$,  set
$$ Q_n : = x^n -1, \  \  R(n) : = \ZZ[x] / (Q_n).$$

We have, setting $\mu_n : = \{ \ze \in \CC | \ze^n=1\}$,
$$ Q_n (x) = \Pi_{\ze \in \mu_n} (x- \ze) ,$$
and we have an irreducible decomposition in $\ZZ[x]$
$$  Q_n (x) = \Pi_{d | n } \Phi_d(x),$$
where $\Phi_d(x)$ is the d-th cyclotomic polynomial
$$ \Phi_d (x) =  \Pi_{\ze \in \mu_n, \ ord(\ze) = d} (x- \ze) .$$

We have (see \cite{lang}, page 280), 
$$\Phi_d(x) =  \Pi_{d' | d }  (x^{d/d'} - 1)^{\mu(d')},$$
where $\mu(d') $ is the M\"obius function, such that
\begin{itemize}
\item
$\mu(d') = 0$ if $d'$ is not square-free
\item
$\mu(1) = 1$
\item
$\mu(p_1 \dots p_r) = (-1)^r$, if $p_1, \dots, p_r$ are distinct primes.

\end{itemize} 

\begin{defin}
Define the {\bf cyclotomic ring} as $R_d : = \ZZ[x] / (\Phi_d)$,
and, for integers $d \neq m,$ $R_{d,m}  : = \ZZ[x] / (\Phi_d , \Phi_m)$.
\end{defin}

\begin{example}\label{n=6}
Consider the polynomial $Q_6 = x^6 - 1= (x^3 -1) (x^3 + 1) = (x-1) (x^2 + x + 1) (x+1)(x^2 -x +1) = \Phi_1 \Phi_3 \Phi_2 \Phi_6$.  

We choose now $d=3, m =6$: then, since $\Phi_6 = \Phi_3 - 2x $, and $ 2( x^2 + x + 1) - 2x (x+1) = 2,$
we obtain that $$( \Phi_3, \Phi_6) = ( x^2 + x + 1, 2x) = ( x^2 + x + 1, 2).$$

Hence $R_{3,6} = \ZZ/2[x] / (x^2 + x + 1) = \FF_4$. 

While $\ZZ / (( \Phi_3, \Phi_6)\cap \ZZ )= \ZZ/2$.

Note that $r := Res_x ( \Phi_3, \Phi_6)$ equals, by the interpolation formula, if $\ze$ is a primitive third root of $1$,
$r = (\ze^2 - \ze + 1)(\ze - \ze^2 + 1) = (- 2 \ze ) (-2 \ze^2) = 4 = | R_{3,6} |$.

Instead, easily we get $R_{1,3} = \ZZ/3$, $R_{1,6} = 0$, $R_{1,2} = \ZZ/2$, $R_{2,3} = 0$, $ R_{2,6} =  \ZZ/3.$
\end{example}

\bigskip

Observe now in general that, since $\Phi_d , \Phi_m$ are monic polynomials, and both irreducible,  their resultant is a non zero integer 
$r = r_{d,m}$,  such  that, if $$\ZZ / (( \Phi_m, \Phi_d)\cap \ZZ ) = \ZZ / (\be_{d,m}),$$  then $ \be_{d,m} |  r_{d,m}$ and the two numbers have the same 
radical.

It is easy to see that $R_{d,m}= 0$ if $d,m$ are relatively prime: since then in the quotient we have $x^m  \equiv 1, x^d  \equiv 1 \Rightarrow (x-1) \equiv 0, $ hence 
$m \equiv   0  \Rightarrow R_{d,m}= 0$.

It is straightforward to calculate the discriminant of $Q_n = x^n-1$ as the resultant of $Q_n$ and its derivative:
$Disc (Q_n) = n^n $. 

However, up to $\pm 1$, $Disc (Q_n) = \Pi_{0 \leq i < j \leq n-1}  ( \e_n^i - \e_n^j )$,
where $\e_n : = exp ( 2 \pi i /n)$.

Since $Q_n = \Pi_{d | n} \Phi_d$, follows that 
$$ n^n = Disc (Q_n) =  \Pi_ {d,m | n , \ d< m} Res (\Phi_d, \Phi_m) \Pi_{d | n} Disc(\Phi_d).$$

The clever calculation of all the factors of the above product was found by Emma  Lehmer \cite{lehmer} in 1930
(in her terminology a simple integer is what is today called a square-free integer)
and then reproven with different proofs by several authors \cite{apostol}, \cite{diederichsen}, \cite{dresden}; in particular, 
the calculation of $ \be_{d,m}$ can be found in an article \cite{dresden}  by G. Dresden.

We summarize these results by Lehmer, Diederichsen, Apostol,  Dresden with a minor addition (here $\phi(d)$ is the Euler function, i.e. , $\phi(d) = \deg \Phi_d$):

\begin{thm} \label{cyclotomicRes}
Let $d, m \in \ZZ, d  <  m $. Then, if $ \be_{d,m}$  is defined by:
$$\ZZ / (( \Phi_m, \Phi_d)\cap \ZZ ) = \ZZ / (\be_{d,m}),$$
then $ \be_{d,m}$ is $\pm 1$ unless $ d | m$ and there exists a prime $p$ such that
$$m = p^k d, $$
in which case $ \be_{d,p^k d} = p$. 

Moreover, $Res_x ( \Phi_d, \Phi_{p^k d}) = p ^ {\phi(d)}$ in the latter case, and $1$ otherwise.

In particular, $R_{m,d} = 0$ unless $ d | m$ and there exists a prime $p$ such that
$m = p^k d, $ and in this case  $R_{d,p^k d}$ is a direct sum of  finite fields $\FF_{p^{\nu}}$ if and only if $ d$ is not divisible by $p$. 

Moreover, $R_{d,p^k d}$ is a field $\FF_{p^{\phi(d)}}$  if and only if the class of $ p$ generates the group $(\ZZ/d)^*$.

\end{thm}

\begin{proof}
Only the last assertions need to be proven, since the rest  is  contained in the cited articles.

Clearly $R_{m,d} = 0$ if $ \be_{d,m}$ is $\pm 1$, since then $ 1 \equiv 0$. 

If instead $ \be_{d,m} = p$,
$R_{d,p^k d}$ is an $\FF_p = \ZZ/p$ module, and 
$$R_{d,p^k d} = \FF_p [x] (\Phi_d, \Phi_{p^k d}) =
\FF_p [x] / (P),$$ where $P$ is the G.C.D. of (the reductions $\Psi_d, \Psi_{p^k d}$ of) $\Phi_d, \Phi_{p^k d}$ inside $\FF_p [x] $.       

(i) If $d$  is not divisible by $p$, then the polynomial $x^d -1$ is square free, and $\FF_p [x] / (x^d-1)$
is a direct sum of fields.  A fortiori $\FF_p [x] / (P)$ is a direct sum of fields. Indeed we shall show next that $P = \Psi_d$.

(ii)  The next question is to show that  $\Psi_d | \Psi_{p^k d}$,
so that $P = \Psi_d$ for $d$  not divisible by $p$.

We use the previously cited formula for the cyclotomic polynomial $\Phi_D$ when $D = p^k d$ and $p$ does not divide $d$
 (using the fact that only the  terms with $D'$ square-free occur):

$$\Phi_D(x) =  \Pi_{D' | D }  (x^{D/D'} - 1)^{\mu(D')} =$$
$$   =  \Pi_{d' | d }  (x^{d p^k/d'} - 1)^{\mu(d')}   \Pi_{d' | d }  (x^{d p^{k-1}/d'} - 1)^{- \mu(d')}.$$
From this we derive, reducing modulo $p$,
$$ \Psi_D(x) =  \Psi_{p^k d}(x) = \Psi_d (x) ^{(p^k - p^{k-1})} = \Psi_d (x) ^{(p-1)  p^{k-1}} .$$ 

(iii) In the case of $D = p^k d$, and where $p | d$, we write $ d = d_1 p^h$, with $d_1$ not divisible by $p$.

The formulae we have just established in (ii) imply 

$$   \Psi_{p^k d}(x) =  \Psi_{p^{k + h}  d}(x) =\Psi_{d_1} (x) ^{(p^{k+h -1} (p-1) )} ,   $$
$$   \Psi_d (x) = \Psi_{p^h d_1}(x) =\Psi_{d_1} (x) ^{(p^{h -1} (p-1) )},$$ 
hence again the  G.C.D. $P$ equals $   \Psi_d (x) = \Psi_{d_1} (x) ^{(p^{h -1} (p-1) )},$
and $\FF_p [x] / (P)$ is an algebra with nilpotents.

(iv) Finally, remains to answer the question: when is the reduction  $\Psi_d$  irreducible?
Certainly not in the case where $ d | (p-1)$ and  $\Psi_d$  then 
splits as a product of linear factors (e.g. $\Psi_4 = (x^2 +1) = (x+2)(x+3) \in \FF_5[x]$).

In general, consider the splitting field of $(x^d -1)$ as an extension of $\FF_p$. It will be the smallest $\FF_{p^k}$
which contains the $d$-th roots of $1$, hence $k$  shall be the smallest integer such that $ d | p^k-1 \Leftrightarrow 
p^k \equiv 1 ( mod \ d).$

Hence $\Psi_d$  irreducible iff the splitting field has degree $k = \phi(d)$, equivalently $p$ is a generator
of the group $(\ZZ/d)^*$.

\end{proof}

\begin{ex}
(i) Consider $R_{4,8}$. It equals the algebra with nilpotents
 $$ \ZZ[x] / (x^2 +1 , x^4 + 1) =  \ZZ[x] / (x^2 +1 , - x^2 + 1) =  \ZZ[x] / (x^2 +1 , 2) = \FF_2[x] / (1+x)^2.$$

\end{ex}

\section{ $R(n) : = \ZZ[C_n] = \ZZ[x] / (x^n -1) $-modules which are torsion free Abelian groups.}

In this section $C_n$ denotes the cyclic group with $n$ elements, $C_n \cong \ZZ/n$. 

Hence the group algebra
$ \ZZ[C_n]$ is isomorphic to 
$$R(n) : = \ZZ[x] / (Q_n) = \ZZ[x] / (x^n -1) $$
 and we can apply the results of the previous section.

Let $\Lam$ be an $R(n)$-module, and assume that $\Lam$ is a finitely generated torsion free Abelian group.

This hypothesis allows us to view $\Lam$ as a lattice in the $\QQ$-vector space $\Lam \otimes \QQ$,
which is therefore also an $R(n)$-module, and an  $R(n) \otimes \QQ$-module.

Since $$R(n) \otimes \QQ = \QQ[x] / (Q_n) = \QQ[x] / (\Pi_{d|n} \Phi_d) = \oplus _{d|n} \QQ[x] / (\Phi_d) =  \oplus _{d|n} R_d \otimes \QQ, $$
and accordingly (see for instance Lemma 24, page 313 of \cite{topmethods}) we have a splitting

$$\Lam \otimes \QQ =  \oplus _{d|n} \Lam_{d, \QQ},$$
where $ \Lam_{d, \QQ}$ is an  $R_d \otimes \QQ$-module, and an $R(n) \otimes \QQ$-module via the projection 
$ R(n) \otimes \QQ \ra R_d \otimes \QQ$.

\begin{defin}\label{summands}
We define $\Lam_d : = \Lam \cap  \Lam_{d, \QQ}$. It is a lattice in $ \Lam_{d, \QQ}$, so that we have an exact sequence 
$$0 \ra  \oplus _{d|n} \Lam_d \ra \Lam \ra  \Lam^0  \ra 0 ,$$
where $\Lam^0$ is a finite Abelian group.

$\Lam_d $ is an $R_d$-module in view of the exact sequence established in section 2:

$$ 0 \ra R(n) \ra \oplus _{d|n} R_d \ra \oplus _{d_1 <  d_2} R_{d_1, d_2}  \ra 0,$$
which shows that $R(n)$ acts on $ \Lam_{d, \QQ}$ via the homomorphism $R(n) \ra R_d$.

\end{defin}

We can make the geometry of the above exact sequences more transparent if we introduce the associated real tori.

\begin{defin}\label{tori}
Given an $R(n)$-module $\Lam$, which is a  lattice, i.e., a free Abelian group of finite rank, we define the associated tori as:
\begin{enumerate}
\item
$A : = (\Lam \otimes \RR )/ \Lam$, and since 
\item
$\Lam \otimes \RR =  \oplus _{d|n} \Lam_{d, \RR},$  we define
\item
$A_d : = \Lam_{d, \RR}/ \Lam_d$, hence
\item
we have an exact sequence
$$ 0 \ra \Lam^0 \ra   \oplus _{d|n} A_d \ra A \ra 0,$$
identifying the cokernel $\Lam^0$ as a finite subgroup of the product torus $A' : =  \oplus _{d|n} A_d$, isogenous to $A$,
\item 
$R'(n) : =  \oplus _{d|n} R_d $ acts on $A'$, and we set
\item
$R^0(n) : =  \oplus _{d_1 < d_2} R_{d_1, d_2}$, so that
\item
we have the exact sequence
$0 \ra R(n) \ra R'(n) \ra R^0(n) \ra 0.$
\end{enumerate}

\end{defin}

\begin{prop}\label{recipe}
The datum of an $R(n)$  module which is a  lattice, i.e., a free Abelian group of finite rank, is equivalent to the datum of 
torsion free $R_d$-modules $\Lam_d$ of finite rank,
and of a finite subgroup $\Lam^0 \subset A' : =   \oplus _{d|n} A_d$, where $A_d : = \Lam_{d, \RR}/ \Lam_d$,
such that:

(A) $\Lam^0$ is stable for the subring $R(n)$, which is equivalent to the requirement:  $ x \Lam^0 = \Lam^0  ( \Leftrightarrow x \Lam^0 \subset \Lam^0)$

(B) $\Lam^0 \cap A_d = 0$ $\forall d|n$.

Properties (A) and (B) imply:

(C) $ \frac{Q_n}{\Phi_d} (\la) \in \Lam_d , \  \forall d | n,  \forall \la \in \Lam$; 

hence, writing an element of  $\Lam^0 $ as 
$(\la_d)_{d | n}$, we have 
$$\la_d \in ( \frac{\Phi_d}{Q_n} \Lam_d ) / \Lam_d   \cong  
 \Lam_d /(\frac{Q_n}{\Phi_d}), $$
and it follows that
 
 (D) The number of such finite subgroups $\Lam^0$ is finite.
\end{prop}

\begin{proof}
Clearly, $\Lam$ is determined by the subgroup $\Lam^0 \subset A'$, and the property that $\Lam$ is stable for the subring $R(n)$ is
equivalent to  property (A)  that $R(n)$ stabilizes $\Lam^0$.

Property (B) ensures that $\Lam \cap  \Lam_{d, \RR} = \Lam_d$.

Property (C) follows right away since $ \frac{Q_n}{\Phi_d} (\la)  \in \Lam$, but its components in $\Lam_{d', \QQ}$
are $=0$ for $d' \neq d$, hence this element lies in $\Lam_d$. 

Property (D) follows since $\Lam^0 \subset \oplus_{d | n}  ( \frac{\Phi_d}{Q_n} \Lam_d ) / \Lam_d $, which is finite group since 
$ ( \frac{\Phi_d}{Q_n} \Lam_d )  / \Lam_d \cong  \Lam_d /(\frac{Q_n}{\Phi_d})$ is a finite module over the finite  ring 
$\ZZ[x] / (\Phi_d, \frac{Q_n}{\Phi_d})$ that we have been describing in the previous section.

\end{proof}

The previous proposition is particularly useful in the case where the Dedekind ring $R_d$ is a PID (Principal Ideal Domain),
because then every torsion  free $R_d$-module is free.

In fact, more generally (see \cite{milnor}) every torsion free module over a Dedekind domain $R$ is the direct sum of a free module with an ideal $I$,
hence  $R$ is a PID iff every torsion free module is free.

 However, how does  the above description of $R(n)$-modules apply to the free module $R(n)$?
 
 The answer is related to finding the inverse map in
 the Chinese remainder theorem, of which  theorem \ref{Factorial} is a generalization. 
 
 \begin{prop}
 
 Consider the following module-homomorphism 
 
 $$ j :   R'(n) = \oplus_{d|n} R_d  \ra R(n),$$
 
 such that $j|_{R_d}$ is  induced  by multiplication with $ Q_n / \Phi_d$ (recall that $R_d = \ZZ[x]/ (\Phi_d)$).
 
 1)  Composing with the natural inclusion $ i : R(n) \ra R'(n)$  we obtain an injective  map:
 $$ \psi :  R'(n) = \oplus_{d|n} R_d  \ra  R'(n) = \oplus_{d|n} R_d ,$$
 which is of diagonal form.   
 
 2) We have $ j (R_d) = R(n) \cap (R_d \otimes \QQ) \subset R(n) \otimes \QQ.$
 
 \end{prop}
\begin{proof}
1) means that  $\psi (R_d) \subset R_d$, which follows since $ Q_n / \Phi_d \equiv 0 \in R_{d'}$ for $d' \neq d$. 

2) follows since $R(n) \cap (R_d \otimes \QQ)$ is the kernel of $R(n) \ra \oplus_{d' | n, d' \neq d} R_{d'}$,
hence it is an $R(n)$-module, hence an ideal in $R(n)$: and it must be the ideal generated by $\Pi_{d' | n, d' \neq d} \Phi_{d'}= Q_n / \Phi_d$.

\end{proof}

 We also notice that $ \psi \otimes \QQ$ is an isomorphism, and that we have a surjection
 $$  Coker (\psi) \ra Coker (i)$$
 in view of the injective maps
 $$ j :  R'(n) \ra R(n), i : R(n)  \ra R'(n)  $$
 whose composition is $\psi$.
 
 $  Coker (\psi)$ is essentially the double of $   Coker (i)$, since 
 $$ Coker (\psi) = \oplus_{d|n} R_d /( Q_n / \Phi_d ) = \oplus_{d|n} \ZZ[x] /( \Phi_d, Q_n / \Phi_d ) = \oplus_{d|n} \ZZ[x] /( \Phi_d, \Pi_{d' \neq d} \Phi_d' ) .$$
 
 Now, $$  R_d /(  \Pi_{d' \neq d} \Phi_d' ) $$
 is by Corollary \ref{cor4} isomorphic to the finite ring
 $$   \oplus_{d'|n , d' \neq d} R_{d,d'}.$$

 Therefore, we have the surjection:
  $$ Coker (\psi) = \oplus_{d|n} R_d /(  \Pi_{d' \neq d} \Phi_d' )  \ra Coker (i) = \oplus_{d,d' | n ,\  d <  d'} R_{d,d'}$$ 
and an exact sequence

$$ 0 \ra Coker (j) \ra Coker (\psi) \ra Coker (i) \ra 0,$$
and this shows that we have an isomorphism $$  Coker (j) \cong  Coker (i)= \oplus_{d,d' | n ,\  d <  d'} R_{d,d'}.$$

We can summarize everything in the following

\begin{prop}\label{direct sum}
We have a sequence  of inclusions:
$$ 0 \ra  M' : =  \oplus_{d|n} ( Q_n / \Phi_d ) R_d  \ra R(n) \ra R' : =  \oplus_{d|n} R_d  ,$$
such that $$R^0(n) : = R(n) / M' \subset R' / M' =  \oplus_{d,d' | n ,\  d \neq   d'}  R_{d,d'} $$
is identified as the submodule of $R' / M'$, 
$$R^0(n) = \{ (a_{d,d'})  | (a_{d,d'}) = (a_{d',d}) \}.$$ 
\end{prop}
\begin{proof}
There remains only to prove the last assertion, which follows immediately from the observation that $R(n)$
is the kernel of the map to $\oplus_{d,d' | n ,\  d <  d'} R_{d,d'}$, given by taking differences $b_d - b_{d'}$
of the image of $b \in R(n)$ to $R_d$.

\end{proof}

\begin{remark}\label{n=6}
Let us go back to example \ref{n=6}, where the divisors of $n=6$ are $1,2,3,6$ and the only nonzero $R_{d,d'}$ 's are: 
\begin{itemize}
\item
$R_{1,2} = \ZZ/2$ acted trivially by $x$, since  $ x \equiv 1$,
\item
$R_{1,3} = \ZZ/3$ acted trivially by $x$, since  $ x \equiv 1$,
\item
$R_{2,6} = \ZZ/3$, where  $x$ acts  multiplying by $-1$, since $ x+1 \equiv 0$,
\item
$R_{3,6} = \FF_4$, with $\ZZ/2$ basis $1,x$ and with $x^2 \equiv  1 + x$.
\end{itemize}
We can now apply the method of proposition \ref{recipe} to construct many $\Lam^0 \subset R' / M' =  \oplus_{d,d' | n ,\  d \neq   d'}  R_{d,d'}.$

For instance, we may take 
$$\Lam^0 = \{ (a_{d,d'}) |   a_{1,2}= a_{2,1}= 0, a_{1,3}= a_{3,1},  a_{2,6}= a_{6,2}, a_{3,6}= a_{6,3} = 0 \},$$
and we get a module different from $R(6)$. 
\end{remark}

\section {Fully ramified cyclic coverings of the projective line and associated Hodge structures}

Let $f : C \ra \PP^1$ be a cyclic covering with Galois group $$\mu_n = \{ \ze \in \CC | \ze^n=1\} \cong  \ZZ/n,$$ branched on $k+2$ points $P_0 = 0 , P_1 =1, P_2, \dots , P_k, \infty$,
and let us  assume that 
 $C$ is the normalization of the affine curve  of equation
$$ y^n = x^{ m_0} (x-1)^{ m_1} (x- t_2)^{ m_2} \dots  (x- t_k)^{ m_k},$$
where  $P_i = \{ x = t_i \}$ and where without loss of generality we may assume $ 1 \leq m_j < n$.

 The local monodromy of the covering around the point $P_i$
sends the standard local generator to the element $ m_i \in \ZZ/n$, hence  the inertia group of $P_i$ is cyclic of order 
$r_i = 
\frac{n}{ G.C.D.(n, m_i)}$.

\begin{defin}
One says that a Galois covering is {\bf fully ramified}  if there is a  branch point whose inverse image consists of only one point.

In the case of curves, this implies that the Galois covering is cyclic with group $ \ZZ/n$,  and there is a branch point $P_i$
with $r_i = n$. 
\end{defin}

In  the following, we shall make the assumption that $f : C \ra \PP^1$ is fully ramified, and without loss of generality 
we may assume  that $m_0 = 1$.

Our goal, in this section, is to describe $H_1(C, \ZZ)$ as an $R(n) = \ZZ [\mu_n]$-module.

Recall that the fundamental group $\pi_1(C)$ is the kernel of the following exact sequence (see for instance \cite{cime}, pages 101-104):

$$ 1 \ra \pi_1(C) \ra \sT : = T (0;  r _0, r_1, \dots, r_k, r_{\infty}) \ra \ZZ/n \ra 0,$$
where the polygonal orbifold group $\sT : = T (0;  r _0, r_1, \dots, r_k, r_{\infty})$ has generators

$$ \ga_0, \ga_1, \dots , \ga_k, \ga_{\infty}$$
and relations
$$ \ga_i^{r_i} = 1, \forall i, \ \ga_0 \cdot  \ga_1\cdot \dots \cdot \ga_k \cdot \ga_{\infty} = 1.$$

 Breaking the symmetry, we shall see  $ \ga_0, \ga_1, \dots \ga_k$ as generators for $\sT$,
 and $ \ga_0^n = 1, \ga_1^{r_1} = 1, \dots ,   \ga_k^{r_k} = 1, ( \ga_0 \cdot  \ga_1\cdot \dots \cdot \ga_k)^{r_{\infty}} = 1$
 as relations.
 
One sees then  immediately that $\pi_1(C)$ is generated by $1 + n (k)$  elements, and, since $\ga_0, \ga_0^2, \dots, \ga_0^{n-1}$
  is a Schreier system, we can choose, by the Reidemeister Schreier method (\cite{mks}, theorem 2.7, page 89 and following) the following generators:
$$ \ga_0^n , \ \de_{i,j} : = \ga_0^i \ga_j \ga_0^{-m_j -i},  i = 0, \dots, n-1, j = 1 , \dots, k.$$

These generators are nice because the Galois group $\ZZ/n$ acts on  $\pi_1(C)$ by conjugation of a lift of
$i \in \ZZ/n$, hence by conjugation by  $\ga_0^i$. Hence these elements are permuted by the Galois group.

It is obvious that we can forget about the first generator $ \ga_0^n$ in view of the relation $ \ga_0^n=1$. 

The Hurwitz formula calculates the genus $g$ of the curve $C$ as follows:
$$ 2g - 2 = n ( -2 + \sum_{i=0, \dots, k, \infty}  \frac{(r_i - 1)}{r_i} ) = \sum_{i=0, \dots, k, \infty} (n -  \frac{n }{r_i} ) - 2n ,$$
hence
$$ 2g  =  \sum_{i=1, \dots, k, \infty} (n -  \frac{n }{r_i} ) - n + 1=  \sum_{i=1, \dots, k} (n -  \frac{n }{r_i} ) + 1 - \frac{n }{r_{\infty}} $$

We shall further reduce the number of generators  using the other relations, until we reach $2g$ generators: the classes
of these in $H_1(C, \ZZ) = \pi_1(C)^{ab}$ will then give a basis for the first homology group.

Indeed, we can rewrite:
$$ 1 = \ga_j^{r_j} =  \ga_0^i \ga_j^{r_j} \ga_0^{-i} = \de_{i,j}  \de_{i + m_j , \ j}  \de_{i + 2 m_j , \ j} \dots  \de_{i + (r_j -1) m_j , \ j}.$$ 

{\bf Case $r_{\infty} = n$:} We can eliminate in this way, since $(m_j ) \subset \ZZ/n$ equals $(n/r_j) \subset \ZZ/n$, 
$\sum_{i=1, \dots, k}  \frac{n }{r_i} $ generators, and we obtain the right number of generators, $ \sum_{i=1, \dots, k} (n -  \frac{n }{r_i} )  = 2g$
generators. The rewriting of the relations coming from $\ga_{\infty}^{r_{\infty}} = 1$ yields the standard relation for $\pi_1(C)$.

\begin{defin}
Define $D_{i,j}$ as the class of  $\de_{i,j}$ inside $H_1(C, \ZZ)$, for $ i= 0, \dots, n-1, $ and $ j=1, \dots k$.

\end{defin}

Then the previous relation rewrites as

$$ \sum_{h=0, \dots, r_j -1} D_{ i + h n/r_j , j} = 0.$$

We have therefore proven

\begin{thm}\label{2totram}
Assume that we have a cyclic covering $ f : C \ra \PP^1$ with group $\ZZ/n$ and with two points of full ramification.

Then, if the order of the inertia groups are $n,r_1, \dots , r_k, n$,  the $\ZZ[x] $-module  $H_1(C, \ZZ)$ is a direct sum of
cyclic  modules, 
$$H_1(C, \ZZ) = \oplus_1^k  ( \ZZ[x] / ( 1 + x^{n/r_j}  + x^{2n/r_j} + \dots + x^{(r_j-1)n/ r_j}) .$$\end{thm}

\begin{ex}\label{domingo}

This example is borrowed from the modular description of the Cartwright-Steger surface, which was 
first  explained to me by Domingo Toledo.

Assume that we have $n=12$, and $ m_0 = 7, m_1=m_2= m_3=2, m_{\infty} = 11$, hence ramification indices 
$(12,6,6,6,12)$. 

Then $ H_1(C, \ZZ) \cong  \oplus_1^6 \ZZ[x] /(x^{10} + x^8 + \dots  + 1).$

Since $x^{10} + x^8 + \dots  + 1 = (x^{12} - 1) / (x^2-1) =  \Phi_3 \Phi_4  \Phi_6 \Phi_{12}$, 
we see that, by the previous results,  setting 
$$M : =  \ZZ[x] /(x^{10} + x^8 + \dots  + 1),$$
we have $ H_1(C, \ZZ) \cong 3 M$, 
and 

$$ M \supset   \Phi_4  \Phi_6 \Phi_{12} M  \oplus   \Phi_3   \Phi_6 \Phi_{12} M  \oplus   \Phi_3   \Phi_4 \Phi_{12} M \oplus \Phi_3 \Phi_4  \Phi_6 M \cong R_3 \oplus R_4 \oplus R_6 \oplus R_{12}.$$ 

An easy calculation shows that the $i$-th dimensional eigenspace $V_j : = H^0(\Omega^1_C)^j$ has dimension $v(j)$ with: 
$$v(j)=  0,{\rm for} \  j=0,5,6, \ v(j)=  2, {\rm for} \ j=1,2,8,9, \ v(j)=  1, {\rm for} \  j=3,4,7,10,11.$$ 

Now, we consider the cohomology $H^1(C, \ZZ)$, again as a $\ZZ[x]$-module. 

Then $H^1(C, \ZZ) = 3 \ Hom( M, \ZZ),$ and by the exact sequence
$$ 0 \ra M \ra \sum_{d=3,4,6,12} M / \Phi_d M = M  \otimes_{\ZZ[x]} R_d = R_d \ra Coker \ra 0,   $$
we find that $Hom( M, \ZZ)$ contains $$\sum_{d=3,4,6,12} Hom (M / \Phi_d M, \ZZ) =  \sum_{d=3,4,6,12} Hom (R_d , \ZZ) \cong  \sum_{d=3,4,6,12} R_d,$$
where the duality $R_d \times R_d \ra \ZZ$ is given by the product followed by the trace map.
\end{ex}

\begin{remark}
Theorem \ref{2totram} works if there are two points of full ramification; if there is exactly one, we are in the situation of 

{\bf Case 2: $r_1, r_2, \dots r_k , r_{\infty} < n$.} 

Here we can  use the relation
$$( \ga_0 \cdot  \ga_1\cdot \dots \cdot \ga_k )^{r_{\infty}} = 1,$$
to  rewrite
$$1 = \ga_0^i ( \ga_0 \cdot  \ga_1\cdot \dots \cdot \ga_k )^{r_{\infty}} \ga_0^{-i} = $$
$$ = \de_{i+1, 1} \cdot  \de_{i+1 + m_1 , 2} \cdot   \dots  \de_{i+1 + m_1 + \dots + m_{k-1} , k} \ \de_{i +1 - m_{\infty}, 1} \dots \de_{i -m_k  , k} =1.$$

These, since $\ZZ/ (m_{\infty}) = \ZZ/ (n/r_{\infty})$,   are $n/r_{\infty}$ relations, as expected.

Passing to the Abelianization of $\pi_1(C)$, we get the extra relations:

$$ \sum_{h=0, \dots r_{\infty}-1} D_{i+1 + h n/r_{\infty},1} +  D_{i+1 + m_1 + h n/r_{\infty},2}  + \dots + D_{ i - m_k +  h n/r_{\infty},k} = 0.$$
This case should be easier to treat than the general one with no points of full ramification.
\end{remark}

\section{Structure Theorem for Bagnera-De Franchis Manifolds}

The goal of this section is to give a complete structure theorem for Bagnera-De Franchis  Manifolds, leaving aside the question of
projectivity, which was treated in \cite{ccd} (in the appendix it was shown that each BdF Manifold deforms to a projective one).

Set in this section $ G = \ZZ/n$ and consider a Bagnera-De Franchis  Manifold $X = A/G$, where $A$ is a complex torus $ A = V / \Lam$
of dimension $g$. Here  we use the letter $A$ even if we have a torus, and not necessarily an Abelian variety, just  in order to have
a similar notation to \cite{topmethods} (where the letter $T$ was used to denote some torsion subgroup).

Holomorphic maps $ F : A \ra A'$ of complex tori are affine maps, since their derivatives in the flat uniformizing parameters are constant:
hence  such holomorphic maps 
$$ F : A = V / \Lam  \ra A' = V' / \Lam' $$
can be represented  as 
$$ F(v) = \al (v) + b \ ( mod \  \Lam') , \al : V \ra V', \al ( \Lam)  \subset  \Lam' , b \in V'.$$
$\al$ is a linear map of vector spaces  induced by $\al |_{\Lam}$, which we still denote by $\al$: indeed, any $\ZZ$-linear map $\al :  ( \Lam)  \ra \Lam' $
induces a complex linear map
$$ \al \otimes \CC :  \Lam \otimes \CC  = V \oplus \overline{V} \ra  \Lam' \otimes \CC  = V' \oplus \overline{V'},$$
and $\al$ induces a homomorphism of complex tori if and only if  
$$ (\al \otimes \CC ) (V) \subset V',$$
i.e. $ \al \otimes \CC$ is a homomorphism of  Hodge structures.

We take now a generator $\ga \in G$, and write $ \ga(v) = \al (v) + b$ (here $A'= A$).

Then we have a decomposition $ V = \oplus_{\ze \in \mu_n} V_{\ze}$, where $V_{\ze}$ is the eigenspace for the complex linear map
$\al$ corresponding to the eigenvalue $\ze$.

The condition that $\ga$ has no fixed point means that there is no solution of the equation
$$ \al (v) + b \equiv v ( mod \Lam) \Leftrightarrow (\al - Id) (v)  + b \in \Lam.$$

Writing $V = V_1 \oplus V_2 , V_2 : =  \oplus_{\ze \in \mu_n, \ze \neq 1} V_{\ze}$, we have that $(\al - Id)$ is invertible on $V_2$,
hence after a change of the origin we may assume that $b \in V_1$, and that
$$ \ga (v_1, v_2 ) = (v_1 + b_1, \al_2 (v_2)).$$
The condition that $G$ operates freely on $A$ amounts to: 
$$ (**) \ \exists \ (\la_1, \la_2)  \in \Lam \  {\rm such \ that} \   h b_1  = \la_1 \Leftrightarrow n | h.$$

$\al$ makes $\Lam$ an $R(n)$ module, hence we have (compare   the notation introduced just before Proposition \ref{recipe})  the decomposition 
$\Lam \otimes \QQ =  \oplus _{d|n} \Lam_{d, \QQ},$  
and setting $$\Lam_d : = \Lam \cap  \Lam_{d, \QQ}, \   \Lam^0  : = \Lam / \oplus _{d|n} \Lam_d , \ A_d : = \Lam_{d, \RR}/ \Lam_d$$

 we have an exact sequence
$$ 0 \ra \Lam^0 \ra   A' : = \oplus _{d|n} A_d \ra A \ra 0.$$

\begin{thm}\label{BdF}

A Bagnera-De Franchis Manifold with group $G = \ZZ/n$  is completely determined by the following data:

\begin{enumerate}
\item
 the datum of 
torsion free $R_d$-modules $\Lam_d$ of finite rank, for all $d | n$, such that  $\Lam_1 \neq 0$ and with $\Lam_1, \Lam_2$ of even rank;
\item
the datum of a finite subgroup $\Lam^0 \subset A' : =   \oplus _{d|n} A_d$, where $A_d : = \Lam_{d, \RR}/ \Lam_d$,
\item
an element $\be_1 \in A_1$  generating a subgroup $\langle \be_1 \rangle$ of order exactly $n$,
such that:
\item
(A) $\Lam^0$ is stable for multiplication by the element $x$ of the subring $R(n) \subset R'(n) : =  \oplus _{d|n} R_d $,
and
\item
(B) $\Lam^0 \cap A_d = 0$ $\forall d|n$,
\item
(C) the projection of $\Lam^0$ into $A_1$   intersects  the subgroup $\langle \be_1 \rangle$ only in $0$;
\item
the datum of a complex structure on each $\Lam_d \otimes \CC $,  i.e., a Hodge decomposition
$$ \Lam_d \otimes \CC  = V(d) \oplus \overline{V(d)}, $$
which allows to decompose $V(d) = \oplus_{j < d, (j,d)=1} V_j$ as a direct sum of eigenspaces for the action $\al$ of $x$.
\item
The properties (A) and (B) imply  that $\Lam^0 \subset \oplus_{d | n}  ( \frac{\Phi_d}{Q_n} \Lam_d ) / \Lam_d $,
hence, in particular, the number of such subgroups $\Lam^0$ is finite.

\end{enumerate}

\end{thm}

\begin{proof}

According to Proposition \ref{recipe} the data (1) and (2), provided that (4) and (5) hold, determine a lattice $\Lam$ which is an $R(n)$-module.

The conditions in (1) that $\Lam_1, \Lam_2$ have  even rank are necessary  for the existence of a complex structure on $ \Lam_d \otimes \CC  $
for $d=1,2$. 

We can then choose  $V_1 = V(1)$ to give a Hodge structure on $\Lam_1 \otimes \CC$ and, if $n$ is even, $ V(2) = V_{n/2}$ 
to give a complex structure on $\Lam_2 \otimes \CC$.

 Whereas, for $ d \geq 3$, $ \Lam_d \otimes \CC$ splits as a direct sum of eigenspaces $W_j$, corresponding to
the eigenvalues $\e_n^{j n/d}$, for $j < d, (j,d) = 1$. 

These eigenvalues come in conjugate pairs, hence it suffices to consider $W_j \oplus W_{n-j}$
and choose $V_j $, for $ j < d / 2$, to be any subspace of $W_j$, letting then $V_{n-j} \subset W_{n-j}$ be a subspace such that 
$W_j = V_j  \oplus \overline{V_{n-j}  }$. 

We are done, since $W_{n-j} = \overline{W_{j}  }$.

Finally, we take the transformation $\ga$ whose linear part is the linear map $\al$ corresponding to multiplication by $x$,
and whose translation part is $
\be_1 \in A_1$. The condition that any power of $\ga$ with exponent $ h <n$ has no fixed points means (see $(**)$) that  
the equation $ (\ga^h (v) - v) \in \Lam$   has no solutions $v \in V$. Let $\be_1$ be the class of $b \in V_1$:
then this equation is equal to
$$h b + (\al^h - Id)(v) \in \Lam.$$
Since the image of  $ (\al^h - Id)$ equals to $V_2$, this means that $h b$ does not belong to the projection
of $\Lam$ into $V_1$. This is equivalent to requiring that $h \be_1$  does not belong to the projection of $\Lam^0$ into $A_1$.

Property (8) was already shown in Proposition \ref{recipe}.

\end{proof}

\begin{remark}

(I) 
To relate the  formulation given here with the content of Proposition 16 of \cite{topmethods}, it suffices to define
$\Lam_2 : = \Lam \cap V_2$, and $T : = \Lam / (\Lam_1 + \Lam_2)$. Then $T$ is isomorphic to the image of
$\Lam^0$ inside $A_1$, which was called $T_1$ in loc. cit. . Hence one requires $T_1$ and $\langle \be_1 \rangle$
to intersect only in $0$, and clearly $ X = A / G = (A_1 \times A_2 )/ G \times T$.

(II) On page 313, eight lines from the bottom  of  \cite{topmethods} there is a `lapsus calami', asserting 
the splitting $ R(m) = \oplus _{d|m} R_d$ without tensoring with $\QQ$. However, fortunately,  this wrong assertion 
is not used at all  in \cite{topmethods}.

\end{remark}

\bigskip

\section{The intersection  product for the homology of fully ramified cyclic coverings of the line}

In this section we use the presentation of the fundamental group of a cyclic covering $ f : C \ra \PP^1$ as described in section 4,
and shall determine the intersection product map $H_1(C, \ZZ) \times H_1(C, \ZZ) \ra \ZZ$ dual to the cup product for the first homology group.

We shall make the assumption that the ramification indices  $r_0 = r_{\infty} = n$.

Then we have a set $\sG_1$  generators for $\pi_1(C)$,  consisting of the $n \cdot k$ elements:  

$$  \de_{i,j} : = \ga_0^i \ga_j \ga_0^{-m_j -i},  i = 0, \dots, n-1, j = 1 , \dots, k.$$

As we saw, the $\sum_{i=1, \dots, k}  \frac{n }{r_i} $ relations:
$$ 1 = \de_{i,j}  \de_{i + m_j , \ j}  \de_{i + 2 m_j , \ j} \dots  \de_{i + (r_j -1) m_j , \ j},$$
allow to eliminate $\sum_{i=1, \dots, k}  \frac{n }{r_i} $ of these generators, 
and we obtain a set $\sG_2$ of generators, with the right number  $ \sum_{i=1, \dots, k} (n -  \frac{n }{r_i} )  = 2g$
of elements.

It is convenient to eliminate the generators $  \de_{i,j} , \ i = 0, \dots  , \frac{n }{r_i} - 1$, and each of them is
then equal to the product 
$$ [ \de_{i + m_j , \ j}  \de_{i + 2 m_j , \ j} \dots  \de_{i + (r_j -1) m_j , \ j}]^{-1} =  \de_{i + (r_j -1) m_j , \ j}^{-1} \dots \de_{i + 2 m_j , \ j}^{-1}   \de_{i + m_j , \ j}^{-1}.$$

Observe that this is the product of exactly $(r_j-1)$ inverses of elements in the  set $\sG_2$ of $2g$  generators.

We consider now the relation which is the Reidemeister-Schreier rewriting of $\ga_{\infty} ^n = 1$.

We have:
$$1 = (\ga_0 \ga_1 \ga_2  \dots \ga_k)^n =$$
$$  (\ga_0 \ga_1\ga_0^{-1})  (\ga_0 \ga_2 \ga_0^{-1}) \dots (\ga_0 \ga_k \ga_0^{-1}) \cdot $$ 
$$  (\ga_0^2 \ga_1\ga_0^{-2})  (\ga_0^2 \ga_2 \ga_0^{-2}) \dots (\ga_0^2 \ga_k \ga_0^{-2}) \cdot $$ 
$$\dots  $$ 
$$  (\ga_0^{n-1} \ga_1\ga_0^{-(n-1)})  (\ga_0^{n-1} \ga_2 \ga_0^{-(n-1)}) \dots (\ga_0^{n-1} \ga_k \ga_0^{-(n-1)}) \cdot $$
$$ \ga_1 \ga_2 \dots \ga_k. $$

We are then ready to rewrite in terms of the `big' set $\sG_1$ of generators:
$$ 1= \de_{ 1, \ 1}  \dots \de_{ 1, \ k } \de_{ 2, \ 1}  \dots \de_{ 2, \ k } \dots \de_{ n-1, \ 1}  \dots \de_{ n-1, \ k } \de_{ 0, \ 1}  \dots \de_{ 0, \ k }.$$

This relation is then equal to the product of the $n \cdot k$ original generators. 

Now, we replace, as indicated above, the generators $  \de_{i,j} , \ i = 0, \dots  \frac{n }{r_i}$ 
by the respective products of inverses of the final set $\sG_2$ of $2g$  generators.

We are then left with a relation where do occur exactly the $2g$ generators in $\sG_2$ 
($2g = n k - \sum_{i=1, \dots, k}  \frac{n }{r_i} $)
with exponent equal to $+1$, and exactly 
$$ \sum_{i=1, \dots, k}  \frac{n }{r_i} (r_i-1) = \sum_{i=1, \dots, k} n -  \frac{n }{r_i}  = 2g$$ 
 different generators  
with exponent equal to $-1$.

Hence the relation, even if not in the standard form, depicts a 2-dimensional manifold obtained attaching a 2-disk to a bouquet of
$2g$ circles.

The recipe for the intersection product is then given in the following

\begin{prop}
Let $\pi_g $ be a group with $2g$ generators, $a_1, \dots a_{2g}$,
 and with only one relation $W (a_1, \dots ,a_{2g} ) = 1$, where the word 
 $W (a_1, \dots ,a_{2g} ) $ consists of exactly $4g$ letters, of which $2g$ are exactly the letters 
 $a_1, \dots a_{2g}$, and $2g$ are exactly the inverses of the letters 
 $a_1, \dots a_{2g}$.
 
 Then  $\pi_g $ is the fundamental group of a closed Riemann surface $C$ of dimension $g$, and the intersection 
 product on $H_1(C, \ZZ)  = \pi_g^{ab} $ is determined as follows, considering the word as giving a cyclical order 
 in the set (of cardinality $4g$) consisting of the generators $a_i$ and of  their inverses:
 \begin{enumerate}
 \item
 $ a_i \cdot a_j  = 0 $ if , removing $a_i$ and $a_i^{-1}$, $a_j$ and $a_j^{-1}$ lie in the same of the two remaining intervals;
 \item
  $ a_i \cdot a_j  = 1 $ if , removing $a_i$ and $a_i^{-1}$, $a_j$ lies in the interval going from $a_i$ to $a_i^{-1}$, and $a_j^{-1}$ lies
  in the other;
  \item
   $ a_i \cdot a_j  = - 1 $ if , removing $a_i$ and $a_i^{-1}$, $a_j^{-1}$ lies in the interval going from $a_i$ to $a_i^{-1}$, and $a_j$  lies
  in the other.

 \end{enumerate}

\end{prop}

\begin{proof}
Consider a bouquet of $2g$ circles meeting in one point $P$, and corresponding to the generators $a_1, \dots a_{2g}$.

We attach a 2-disk whose boundary is the word $W (a_1, \dots ,a_{2g} ) $. Since in the word each generator $a_i$ and 
its inverse $a_i^{-1}$ appear exactly once, the space $C$  that we obtain is 
smooth outside of $P$. At $P$ however we have $4g$ segments coming in, and we fill in $4g$ angles, hence $C$ is a topological
 manifold. It is also oriented since $H_2(C, \ZZ) = \ZZ$.
 
 The two incoming segments corresponding to $a_i$ and $a_i^{-1}$ (who are oriented)  locally separate $C$ in a neighbourhood of $P$.
 
 In case (1), two incoming segments corresponding to $a_j$ and $a_j^{-1}$ lie in the same half-plane, hence they can be deformed 
 out of $P$ until they do not intersect 
 $a_i$ at all. In cases (2) and (3), these segments lie in different halfplanes, hence the intersection product equals $\pm1$.
 
 In case (2) the intersection is positively oriented, in case (3) it is negatively oriented, as one sees easily (compare the standard
 presentation where the word $W (a_1, \dots ,a_{2g} ) = a_1 a_2 a_1^{-1} a_2^{-1} \dots$).

\end{proof}

\begin{ex}

Consider the Fermat elliptic curve $E$ 
of affine equation
$$ y^3 = x (x -1).$$

Here, $n=3$ and $k=1$, and we have generators 
$$ \de_{ 0, \ 1} = \ga_1 \ga_0^{-1},  \ \de_{ 1, \ 1} = \ga_0 \ga_1 \ga_0^{-2} , \ \de_{ 2, \ 1}  = \ga_0^2\ga_1,$$
satisfying 
$$\de_{ 0, \ 1}  \de_{ 1, \ 1}  \de_{ 2, \ 1}  = 1, \  \de_{ 1, \ 1}  \de_{ 0, \ 1}  \de_{ 2, \ 1}  = 1.$$ 

Eliminating $\de_{ 2, \ 1} $, we get generators $\de_{ 0, \ 1}  \de_{ 1, \ 1}$ with relation 
$$ \de_{ 0, \ 1}  \de_{ 1, \ 1} = \de_{ 1, \ 1}  \de_{ 0, \ 1} ,$$
as fully expected.

Consider now the example of the curve $C$ of affine equation
$$ y^3 = x (x^4 -1).$$

Here, $n=3$ and $k=4$, and we have generators 
$$\de_{ 0, \ j} , \  \de_{ 1, \ j} , \  \de_{ 2, \ j} , \ j=1,2,3, 4, $$
satisfying relations 
$$\de_{ 0, \ j}  \de_{ 1, \ j}  \de_{ 2, \ j} = 1, \ $$
$$\de_{ 1, \ 1}  \de_{ 1, \ 2}  \de_{ 1, \ 3}  \de_{ 1, \ 4} \de_{ 2, \ 1}  \dots  \de_{ 2, \ 4} \de_{ 0, \ 1}  \dots  \de_{ 0, \ 4} = 1. \ $$

Set $$\de_{ 1, \ j} = : a_j , \ \de_{ 2, \ j} = : b_j.$$

Then we have $8$ generators satisfying the relation
$$ a_1 a_2 a_3 a_4 b_1 b_2 b_3 b_4 b_1^{-1}  a_1^{-1} b_2^{-1}   a_2^{-1}  b_3^{-1}  a_3 ^{-1}  b_4^{-1}  a_4 ^{-1} = 1.$$ 

Using the rule we just described, we find {\verde (verify once more)} :

$$a_i a_j = a_i b_j = b_i b_j = 1, \ 1 \leq i <  j \leq 4,$$
$$b_1 a_j = 0, j=1, 2,3,4, \ b_2 a_j =0, j=2, 3,4 , \ b_3 a_j = 0,j=3,4 , \ a_4 b_4= 0 .$$
\end{ex}

\begin{remark}
Together with Jong Hae Keum, Matthew Stover and Domingo Toledo, we proved that the Jacobian $J(C)$ of the above curve
$C$ is isomorphic to $E^4$, the product of four copies of the Fermat elliptic curve; and we also disovered  other curves
whose Jacobian is a product (or is isogenous to a product) of elliptic curves.

For the above curve $C$,  however, we  found later that the same example had been described by Ryo Nakajima in \cite{nakajima}.
\end{remark}

{\bf Acknowledgements:} I would like to thank Andreas Demleitner, Jong Hae Keum, Matthew Stover and Domingo Toledo 
for useful conversations.

\end{document}